\documentclass[11pt]{article}
\usepackage[utf8]{inputenc}
\usepackage{graphicx}
\usepackage[letterpaper,top=2.5cm,bottom=2.5cm,left=2.5cm,right=2.5cm, marginparwidth=1.75cm]{geometry}
\usepackage{subfiles}
\usepackage{import}
\usepackage{setspace}
\usepackage{tabu}
\usepackage{appendix}
\usepackage[]{algorithm2e}
\usepackage{graphicx}
\usepackage{authblk}

\graphicspath{ {./images/} }

\usepackage{amsmath}
\usepackage{amsthm}
\usepackage{amssymb}
\numberwithin{equation}{section}
\usepackage{sectsty}
\allsectionsfont{\sffamily}

\newtheorem{theorem}{Theorem}[section]
\newtheorem*{theorem*}{Theorem}
\newtheorem{corollary}[theorem]{Corollary}
\newtheorem{lemma}[theorem]{Lemma}
\newtheorem*{lemma*}{Lemma}
\newtheorem{prop}[theorem]{Proposition}
\theoremstyle{definition}
\newtheorem*{remark}{Remark}
\newtheorem{definition}[theorem]{Definition}
\newtheorem{example}[theorem]{Example}

\begin{document}

\title{Analogues of the Robin-Lagarias Criteria for the Riemann Hypothesis}
\author[1]{Lawrence C. Washington}
\author[2]{Ambrose Yang}
\affil[1]{Department of Mathematics, University of Maryland}
\affil[2]{Montgomery Blair High School}
\maketitle
\begin{abstract}
Robin's criterion states that the Riemann hypothesis is equivalent to $\sigma(n) < e^\gamma n \log\log n$ for all integers $n \geq 5041$, where $\sigma(n)$ is the sum of divisors of $n$ and $\gamma$ is the Euler-Mascheroni constant.
We prove that  the Riemann hypothesis is equivalent to the statement that $\sigma(n) < \frac{e^\gamma}{2} n \log\log n$ for all odd numbers $n \geq 3^4 \cdot 5^3 \cdot 7^2 \cdot 11 \cdots 67$. Lagarias's criterion for the Riemann hypothesis states that the Riemann hypothesis is equivalent to $\sigma(n) < H_n + \exp{H_n}\log{H_n}$ for all integers $n \geq 1$, where $H_n$ is the $n$th harmonic number. We establish an analogue to Lagarias's criterion for the Riemann hypothesis by creating a new harmonic series $H^\prime_n = 2H_n - H_{2n}$ and demonstrating that the Riemann hypothesis is equivalent to $\sigma(n) \leq \frac{3n}{\log{n}} + \exp{H^\prime_n}\log{H^\prime_n}$ for all odd $n \geq 3$. We prove stronger analogues to Robin's inequality for odd squarefree numbers. Furthermore, we find a general formula that studies the effect of the prime factorization of $n$ and its behavior in Robin's inequality.
\end{abstract}

\section{Introduction}
The Riemann hypothesis, an unproven conjecture formulated by Bernhard Riemann in 1859, states that all non-real zeroes of the Riemann zeta function $$\zeta(s) = \sum_{n=1}^{\infty} \frac{1}{n^s}$$ have real part $\frac{1}{2}$. In 1984, Robin \cite{13} (see also [4])  showed that the Riemann hypothesis is equivalent to the statement that 
\begin{equation}\label{Robineq}
f(n) = \frac{\sigma(n)}{n\log{\log{n}}} < e^\gamma
\end{equation}
for all integers $n > 5040$, where $\sigma(n)$ is the sum of divisors function and $\gamma$ is the Euler-Mascheroni constant. In this paper $\log$ is assumed to be the natural logarithm.

In \cite{14}, Robin considered primes in arithmetic progressions. He reformulated his original inequality in terms of 
$$\sigma_{k, l}(n) = \prod_{p \equiv l(k)}\sigma(p^e)$$ 
where $n = \prod_{p | n} p^e$. He also defined an analogue of his $f(n)$ function: 
$$f_{k,l}(n) = \frac{\sigma_{k, l}(n)}{n(\log{(\phi(k)\log{n}}))^{\frac{1}{\phi(k)}}}$$ where $\phi(k)$ is Euler's totient function. Note that $f_{2,1}(n) = f(n)$ for odd $n$. 
Let $\lim \sup f_{k, l}(n) = \alpha_{k, l}$.
The main theorem in \cite{14}, specialized to arithmetic sequences mod 4,  states the following:

\begin{theorem} {\rm \cite{14}}
(a) There are infinitely many $n$ such that $f_{4,1}(n) > \alpha_{4,1}$.\newline
(b) If the Generalized Riemann hypothesis is true for the zeta function of $\mathbb Q(i)$, then $f_{4,3}(n) < \alpha_{4,3}$ when $n$ is sufficiently large. \newline
\end{theorem}
\noindent Robin also gave the following formulas:
$$\alpha_{4,1} = \left(\frac{\pi e^\gamma}{8} \prod_{p \equiv 3(4)} \left(1-\frac{1}{p^2}\right)\right)^{\frac{1}{2}} \approx 0.7738\cdots$$
$$\alpha_{4,3} = \frac{e^\gamma}{2\alpha_{4,1}}\approx 1.1508\cdots.$$
(These numerical values differ from those given by Robin, which seem to be slightly inaccurate. We thank Pieter Moree for pointing out that these constants can be evaluated to many decimal places and that slight variants of them occurred in Ramanujan's letters to Hardy. See
\cite[Section 9]{8} and [11].)
From Robin's work, we ask how $f(n)$ behaves with primes 1 mod 4 and 3 mod 4 working in concert. Under the Generalized Riemann hypothesis,  $f_{4,3}(n)$ and $f(n)$ 
eventually stay below their limits superior, while $f_{4,1}(n)$ exceeds its limit superior infinitely often. It is natural to ask how $f_{2, 1}(n)$ behaves: does it pass over its limit superior infinitely often, or does it stay under the limit superior after a cutoff point.  By \cite{6}, we have that $\alpha_{2,1} = \frac{e^\gamma}{2}$. In this paper we prove that eventually $f_{2, 1}(n)$ does stay under its limit superior, assuming the Riemann hypothesis. However, the conflicting forces of 1 mod 4 and 3 mod 4 primes yield a very large cutoff point. 
The main theorem of the present paper is the following analogue of Robin's result:
\begin{theorem}\label{MainThm1}
The Riemann hypothesis is equivalent to the statement that 
\begin{equation}\label{RobinOdd}
\frac{\sigma(n)}{n\log{\log{n}}} < \frac{e^\gamma}{2}
\end{equation}
for all odd $n \ge c_0 = 18565284664427130919514350125 = 3^4 \cdot 5^3 \cdot 7^2 \cdot 11\cdot 13 \cdots 67$.
\end{theorem}

In 2002, Lagarias, using the results of \cite{13}, showed that the Riemann hypothesis is equivalent to the statement $$\sigma(n) \leq H_n + \exp(H_n) \log(H_n)$$ for $n \geq 1$, where $H_n = \sum_{k=1}^n \frac{1}{k}$. We formulate an analogue of Lagarias's criterion for odd $n$ by first defining a new harmonic series $$H^\prime_n = 2\sum_{k=1}^n \frac{1}{k} - \sum_{k=1}^{2n} \frac{1}{k} = 2H_n-H_{2n}.$$
Our analogue to Lagarias's criterion is the following theorem.
\begin{theorem}\label{MainThm2}
The Riemann hypothesis is equivalent to the statement that
$$\sigma(n) \leq \frac{3n}{\log{n}} + \exp{H^\prime_n}\log{H^\prime_n}$$
for all odd $n \geq 3$.
\end{theorem}

We note that the main term in Lagarias's result is $\exp(H_n) \log(H_n)$; the addition of $H_n$ is simply to account for small $n$. Analogously, our addition of $\frac{3n}{\log n}$ is also to account for small $n$.

The theme of Robin's inequality occurs in several places in the literature. In particular,  Robin's inequality has been studied from various points of view and  for certain
families of integers. See, for example, \cite{2}, \cite{5}, \cite{12}, \cite{18}.

\subsection*{Outline}
In Section 2 we introduce the concept of odd colossally abundant numbers, which will be used to prove our analogue of Robin's inequality in Section 3. In Section 4 we prove our analogue to Lagarias's inequality. Section 5 shows that our analogue (\ref{RobinOdd}) of Robin's inequality is true for odd squarefree numbers. A study of how the prime factorization of $n$ affects the behavior of $f(n)$ in Robin's original inequality (\ref{Robineq}) is also contained in Section 5. The computations
needed in the proofs of Proposition \ref{robinProp} and Theorem \ref{lagariasThm} are contained in Section 6.

\section{Odd Colossally Abundant Numbers}
In the study of the criteria of Robin and Lagarias for the Riemann hypothesis, we are especially concerned with integers $n$ for which $\frac{\sigma(n)}{n}$ is large. In particular, we need a class of integers known as \textit{colossally abundant numbers} and a closely related family of \textit{odd colossally abundant numbers}.

\begin{definition}\label{MainDef} An integer $M\ge 2$ is \textit{colossally abundant} if there exists $\epsilon>0$ such that 
 $$\frac{\sigma(k)}{k^{1+\epsilon}} < \frac{\sigma(M)}{M^{1+\epsilon}} \quad \text{for all $k$ with } 2\le k <M$$
and
 $$\frac{\sigma(M)}{M^{1+\epsilon}} \ge \frac{\sigma(k)}{k^{1+\epsilon}} \quad \text{for all } k >M.$$
An odd integer $N\ge 3$ is \textit{odd colossally abundant} if there exists $\epsilon>0$ such that 
 $$\frac{\sigma(k)}{k^{1+\epsilon}} < \frac{\sigma(N)}{N^{1+\epsilon}} \quad \text{for all odd $k$ with } 3\le k <N$$
and
 $$\frac{\sigma(N)}{N^{1+\epsilon}} \ge \frac{\sigma(k)}{k^{1+\epsilon}} \quad \text{for all odd } k >N.$$
({\it Note:} An odd colossally abundant number is not the same as a colossally abundant odd number since all colossally abundant numbers are even. Throughout this paper, for the convenience of the reader, we use $M$ to denote colossally abundant numbers and $N$ to denote odd colossally abundant numbers.)
\end{definition}

Essentially, a colossally abundant number $M$ is an integer for which the function $\frac{\sigma(k)}{k^{1+\epsilon}}$ attains its maximum
and similarly for odd colossally abundant. Some authors use the non-strict inequality
$\frac{\sigma(k)}{k^{1+\epsilon}} \le \frac{\sigma(M)}{M^{1+\epsilon}} \quad \text{for } 1\le k <M$. This allows some values of $\epsilon$ to correspond to more than one colossally
abundant number and therefore such values of $\epsilon$ need to be removed from consideration in some situations such as Lemma \ref{AlphaLemma} below. To avoid this, we use the definition given on page 455 (not page 448, where the inequalities are switched) of \cite{1}. 
However, we use colossally abundant numbers only as a tool in the proofs of our theorems and the
distinction between the two definitions does not affect our proofs.

Colossally abundant numbers were studied by Alaoglu and Erd\H{o}s in 1944, and they derived an explicit formula for finding a colossally abundant number given a fixed $\epsilon$. 
Since we exclude $M, N=1$ from being (odd or standard) colossally abundant, we need to give upper bounds on $\epsilon$ in the following lemma in order to ensure that $M, N>1$.
In the remainder of the paper, we implicitly assume, without further mention,  that $\epsilon$ is less than these upper bounds. 
 \begin{lemma}\label{AlphaLemma} {\rm \cite{1}} Let $0<\epsilon \le -1+\log_2(3)$. If $M$ is colossally abundant corresponding to $\epsilon$, then 
$$M = \prod{p^{\alpha_p(\epsilon)}},$$
where
\begin{equation}
    \alpha_p(\epsilon) = \left\lfloor\frac{\log{(p^{1+\epsilon}-1) - \log(p^\epsilon - 1)}}{\log{p}}\right\rfloor - 1, \label{expCA}
\end{equation}
and every number given by this formula is colossally abundant. Similarly, let $0<\epsilon\le -1+\log_3(4)$. If $N$ is odd colossally abundant  corresponding to $\epsilon$, then
$$N = \prod_{p \text{ prime }\ge 3} {p^{\alpha_p(\epsilon)}},$$
and every $N\ge 3$ of this form is odd colossally abundant.
\end{lemma}
\begin{proof} 
Let $g(M) = \frac{\sigma(M)}{M^{1+\epsilon}}$. Since $g$ is multiplicative,
$$g(M) = \frac{\sigma(M)}{M^{1+\epsilon}} = \prod_{p^e||M}{\frac{\sigma(p^e)}{p^{e(1+\epsilon)}}} = \prod_{p^e||M}{g(p^e)}.$$
Maximizing $g$ is equivalent to maximizing each term in the product (and always taking the smaller power of a prime $p$ if $\epsilon$ yields $g(p^a)=g(p^b)$ for some $a, b$). 
It is proved in \cite{1} that the optimal exponent for each prime $p$ is given by the formula $\alpha_p(\epsilon)$ above. For odd numbers, we must maximize each term in the product as well, simply leaving out $p = 2$. 
\end{proof}

\begin{definition} {\rm \cite{13}} \label{fEpsilon}
For $x > 1$ and an integer $k \geq 1$, define 
$$F(x, k) = \frac{\log{\left(1+\frac{1}{x+x^2+\dots+x^k}\right)}}{\log{x}} = \frac{\log\left(\frac{x^{k+1}-1}{x(x^k-1)}\right)}{\log x}.$$ 
\end{definition}

\begin{definition} {\rm \cite{13}} \label{expLabel}
For $\epsilon>0$ and an integer $k\ge 1$, let $x_k$ be the value of $x$ such that $F(x, k) = \epsilon$. 
(Since $F(x,k)$ strictly decreases from ${\infty}$ to 0 for $x\in (1,\infty)$, the number $x_k$ is uniquely determined.)
\end{definition}

The following gives the significance of the numbers $x_k$. We include the proof from \cite{7} in order to show the relevance of the function $F$.
\begin{lemma} \label{xkLemma1}{\rm \cite[p. 70]{7}}
Let $M$ or $N$ be the (odd or even) colossally abundant number generated by $\epsilon$. Then $\alpha_p(\epsilon) = k$ for $x_{k+1} < p \le x_k$.
\end{lemma}
\begin{proof} Since $F(x,k)$ is a strictly decreasing function for each $k$, the inequality $x_{k+1}< p \le x_k$  is equivalent to
$$
F(p, k+1) < F(x_{k+1}, k+1)=\epsilon = F(x_k, k)  \le F(p, k),
$$
which is the same as
$$
\frac{p^{k+2}-1}{p^{k+1}-1} < p^{1+\epsilon} \le \frac{p^{k+1}-1}{p^k-1}.
$$
This can be rearranged to
$$
p^{k+1}\le \frac{p^{1+\epsilon}-1}{p^{\epsilon}-1}< p^{k+2},
$$
which is equivalent to $\alpha_p(\epsilon)=k$. \end{proof}

\begin{lemma}\label{xkLemma2}{\rm \cite[p. 190]{13}} (a)  $x_k > x_1^{\frac{1}{k}}$ for all $k\ge 2$.
\newline
(b) $\sqrt{2x_1} > x_2 > \sqrt{x_1}$.\newline
(c) If $x_1\ge 1530$, then  
\begin{equation}
x_2 > \sqrt{2x_1} - \frac{\sqrt{2x_1}\log{2}}{2\log{x_1}}. \label{eq26}
\end{equation}
\end{lemma}
\begin{proof} We sketch the proof from \cite{13}. Part (a) follows from the statement that
$F(t,k) > F(t^k,1)$ for $t>1$ by taking $t^k=x_1$. The strict monotonicity of $F(x,k)$ shows that we must have $x_k>t=x_1^{1/k}$ in order to obtain $F(x_k,k)=F(x_1,1)=\epsilon$.
Part (b) is proved similarly.  For part (c), see \cite{13}.\end{proof}

\begin{example}
Let $\epsilon = 0.021$. Using \eqref{expCA}, we compute $M = 4324320 = 2^5 \cdot 3^3 \cdot 5 \cdot 7 \cdot 11 \cdot 13$. In fact, there is a range of values of $\epsilon$ that yield the same colossally abundant number because of the floor function in \eqref{expCA}. We have $x_1 = 15.77$, $x_2=4.87$, $x_3 = 3.06$, $x_4=2.38$, and $x_5 = 2.03$. Thus, all primes $x_2 < p \le x_1$ have exponent 1, while all primes $x_3 < p \le x_2$ have exponent 2, and so on.
\end{example}

A list of colossally abundant numbers up to $10^{7}$ from \cite{1} is provided in Table 1. To obtain odd colossally abundant numbers, simply remove the powers of 2
from the factorizations. A larger table of odd colossally abundant numbers is in Section 6.

\begin{table}[ht]\label{ECA}
\caption{Colossally abundant numbers $< 10^7$.}
\begin{center}{\begin{tabular}{@{}rcl
@{}} \hline
$M$ && Factorization of $M$ \\ 
2 && 2 \\ 
 6 && $2 \cdot 3$ \\
 12 && $2^2 \cdot 3$ \\
 60 && $2^2 \cdot 3 \cdot 5$\\
 120 && $2^3 \cdot 3 \cdot 5$\\
 360 && $2^3 \cdot 3^2 \cdot 5$\\
 2520 && $2^3 \cdot 3^2 \cdot 5 \cdot 7$\\
 5040 && $2^4 \cdot 3^2 \cdot 5 \cdot 7$\\
 55440 && $2^4 \cdot 3^2 \cdot 5 \cdot 7 \cdot 11$\\
 720720 && $2^4 \cdot 3^2 \cdot 5 \cdot 7 \cdot 11 \cdot 13$\\
 1441440 && $2^5 \cdot 3^2 \cdot 5 \cdot 7 \cdot 11 \cdot 13$\\
 \phantom{aaaaa}4324320 && $2^5 \cdot 3^3 \cdot 5 \cdot 7 \cdot 11 \cdot 13$\phantom{aaaaa}\\
 \hline
\end{tabular}}\end{center}
\end{table}

The exponents in the prime factorizations of (odd or even) colossally abundant numbers are non-increasing as the primes increase, with a long tail of primes to the first power at the end of the factorization. 

In forming colossally abundant numbers, we choose $\epsilon$, and using $\epsilon$ and \eqref{expCA}, we find the colossally abundant number $N$ as well as the exponents of individual primes in $N$.

\begin{lemma} \label{primeEpsilon} For fixed $\epsilon > 0$ and $p$ any prime,  $\alpha_p(\epsilon) = 0$ if and only if $$\epsilon > \log_p{(p+1)} - 1.$$
\end{lemma}
\begin{proof}
$$\alpha_p(\epsilon) = \left\lfloor\frac{\log{(p^{1+\epsilon}-1) - \log(p^\epsilon - 1)}}{\log{p}}\right\rfloor - 1 = 0$$
$$\iff \frac{\log{(p^{1+\epsilon}-1)} - \log(p^\epsilon - 1)}{\log{p}} < 2$$
$$\iff p^{1+\epsilon}(p-1) > p^2-1$$
$$\iff \epsilon > \log_p{(p+1)} - 1.$$ 
\end{proof}

The following allows us to control the difference between a colossally abundant number and the associated odd colossally abundant number obtained by removing the power of $2$.

\begin{lemma}\label{expbound}
In the prime factorization of a colossally abundant number $M$, the exponent of 2 is less than $c\log\log M$ for some constant $c$ independent of $M$.
\end{lemma}
\begin{proof}
Let $\epsilon$ be  a value yielding the colossally composite number $M$, as in Definition \ref{MainDef}. By Lemma \ref{AlphaLemma}, the exponent of 2 is
$$\alpha_2(\epsilon) = \left\lfloor\frac{\log{(2^{1+\epsilon}-1)} - \log{(2^\epsilon - 1)}}{\log{2}}\right\rfloor-1.$$
For $\epsilon \geq 0$, we have $\log{(2^{1+\epsilon}-1)} \leq \left(2\log 2\right)\epsilon$ and $\epsilon\log{2} \leq 2^\epsilon - 1$, since both inequalities have equality at $\epsilon = 0$, and the derivatives of the LHS are less than the derivatives of the RHS for $\epsilon > 0$.
Thus, we have that 
$$\alpha_2(\epsilon) \leq \left\lfloor\frac{ \left(2\log 2\right)\epsilon - \log(\epsilon\log 2)}{\log{2}}\right\rfloor-1.$$
Since $\alpha_2(\epsilon)\ne 0$, Lemma \ref{primeEpsilon} with $p=2$ implies that $\epsilon \le \log_2(3)-1 < 0.6$, so 
$$\alpha_2(\epsilon) \leq -k\log{\epsilon},$$ where $k > 0$ is some constant.
Let $p$ be the largest prime factor dividing $M$. If $q$ is the next prime after $p$, Lemma \ref{xkLemma1} implies that $p < x_1< q$, where $x_1$ is as in 
Definition \ref{expLabel}. By Bertrand's Postulate, there is a prime $q$ with $x_1 < q < 2x_1$. Thus $\alpha_q(\epsilon) = 0$, so by 
Lemma \ref{primeEpsilon}, $$\epsilon > \log_q(q+1) - 1 > \log_{2x_1}(2x_1+1) - 1.$$
Therefore, $\epsilon > \frac{\log\left( 1 + \frac{1}{2x_1} \right)}{\log{2x_1}}$, so
$$\log{\epsilon} > \log{\log{\left( 1+\frac{1}{2x_1} \right)}} - \log{\log{2x_1}} > -3\log{x_1}$$
for $x_1 \geq 2$. Thus $\alpha_2(\epsilon) < c'\log{x_1}$ for some constant $c'$. 

Since $\log M \ge \sum_{p\le x_1} \log p \ge c'' x_1$ for some $c''>0$ (a weak form of the Prime Number Theorem), we 
obtain $\alpha_2(\epsilon) <  c \log\log M$ for some constant $c$.
\end{proof}

\section{Robin's Criterion}

For each odd colossally abundant number $N$, we implicitly choose, once and for all,  an $\epsilon>0$ as in Definition \ref{MainDef}. This yields numbers $x_1>x_2>x_3>\cdots$ as in Definition \ref{expLabel}.
In the remainder of the paper, the only properties we need for the $x_k$ are those given in Lemmas \ref{xkLemma1} and \ref{xkLemma2}. In other words,
these auxiliary numbers tell us which primes have which exponents in the factorization of $N$, and all we need are the estimates on these numbers given by Lemma \ref{xkLemma2}.   The value of $\epsilon$ does not enter into any of the proofs. Therefore, even though the numbers $x_k$ depend on $\epsilon$, we do not include 
$\epsilon$ in the notation.  
However,  $p\mid N$ if and only if $p\le x_1$, so for a given $N$ there is a limited range of values possible for $x_1$, and similarly for the other $x_k$.
A statement such as ``If $x_1\ge 120409$'' in Lemma 3.3 could be replaced by  "If the largest prime factor of $N$ is at least 120409.'' The checking of
smaller cases would then be for $N$ with largest prime factor less than 120409, which is what we do in the proof of Proposition 3.5. But using the numbers $x_k$
is a way to have a much smoother and more coherent framework in which to work.

Throughout this section, we let $N$ be a fixed odd colossally abundant number with the chosen  $x_1, x_2, \dots$.

Let $\theta(x) = \sum_{p\leq x} \log{p}$ be Chebyshev's first function. Let $\theta^\prime(x)$ be the same as $\theta(x)$ but summing over odd primes only.
Then $\theta'(x)=\theta(x)-\log 2$ for $x\ge 2$. 

\begin{lemma}\label{thetaprimeLemma}
Let $x\ge 347$ be a real number and let 
$$
A(x)=\theta'(x)+\theta'(x^{2/3}/2^{1/3}) + \theta'(x^{1/2}/2^{1/4}).
$$
Then $A(x) > .998x$.
\end{lemma}
\begin{proof}
From \cite[p. 265]{16}, if $x\ge 487381$, then $\theta(x)>.998x$. Since $\theta'(x^{2/3}/2^{1/3})> \log 2$,
we have $A(x)\ge \theta'(x)+\theta'(x^{2/3}/2^{1/3}) > .998x$.

From \cite[p. 72]{15}, if $0<x\le 10^8$ then $\theta(x)> x-2.05282x^{1/2}$. Therefore,
$$
A(x)-.998x > .002x + \frac{x^{2/3}}{2^{1/3}} + \frac{x^{1/2}}{2^{1/4}} - 2.05282\left(x^{1/2}+\frac{x^{1/3}}{2^{1/6}} + \frac{x^{1/4}}{2^{1/8}}\right) - \log(8).
$$
The right side of this inequality is a polynomial in $t=x^{1/12}$, and the 12th power of the largest real root $t$ is $558.06$. Therefore, $A(x)-.998x>0$ for $x\ge 559$.
A quick computer calculation shows that $A(n)-.998n\ge 1$ for all integers $n$ with $347\le n\le 559$. If $x$ is a real number with $347\le x\le 559$, write $x=n+\eta$ with $0\le \eta<1$.
Then 
$$A(x)\ge A(n)\ge 1+.998n = 1+.998(x-\eta)> .998x.
$$
This proves the lemma. \end{proof}

\begin{lemma} \label{sumThetaBound} Let $N$ be an odd colossally abundant number with associated numbers $x_1, x_2$.  If $x_2\ge 347$, then 
$$
\log(N) > \theta'(x_1)+.998x_2.
$$
\end{lemma}
\begin{proof}
We have 
\begin{equation} \label{eq9}
    \log{N} = \sum_{i \geq 1} \theta^\prime(x_i).
\end{equation}
We can see this as follows: $\theta^\prime(x_1)$ accounts for all the odd primes up to $x_1$ once, then $\theta^\prime(x_2)$ accounts for all the odd primes up to $x_2$ an additional time, and so on for each $x_i$. Thus we sum all $\log{p}$ an appropriate number of times for each prime.

Lemma \ref{xkLemma2} implies that
\begin{equation}\label{eq3}
    x_1 > \frac{{x_2}^2}{2}.
\end{equation}
and therefore
\begin{equation}
    x_k > \left(\frac{{x_2}^2}{2}\right)^{\frac{1}{k}}.
\end{equation}
for $k\ge 2$.
Using these inequalities, we find that
\begin{align*}
\theta'(x_1)+\theta'(x_2)+\theta'(x_3)+\cdots &\ge \theta'(x_1) + \theta'(x_2)+\theta'(x_2^{2/3}/2^{1/3}) + \theta'(x_2^{1/2}/2^{1/4})\\
&= \theta'(x_1)+A(x_2) > \theta'(x_1)+.998x_2.
\end{align*}
This proves the lemma. \end{proof}

\begin{lemma} \label{logBound}
Let $N$ be odd colossally abundant $N$ with associated $x_1, x_2$. If $x_1\ge 120409$, then 
$$\log{\log{N}} > \log{\theta(x_1)}\exp{\left(\frac{0.99154x_2}{x_1\log{x_1}}\right)}.$$ 
\end{lemma}
\begin{proof}
If $x_1\ge 120409$, then $x_2\ge \sqrt{x_1} \ge 347$. Since $\theta'(x_1)=\theta(x_1)-\log 2$,
$$    \log{N} > \theta(x_1) -\log 2 + 0.998x_2 = \theta(x_1)\left(1+\frac{.998x_2-\log 2}{\theta(x_1)}\right).
$$
Let $y = \frac{0.998x_2 - \log{2}}{\theta(x_1)}$. Using the fact that $\log{(1+x)} \geq \frac{x}{1+x}$ for $x > -1$, we have the following:
\begin{align}
    \log{\log{N}} &\geq \log{\theta(x_1)} + \log{(1+y)} \\
    &\geq \log{\theta(x_1)} + \frac{y}{1+y} \\
    &= (\log{\theta(x_1)})\left(1 + \frac{y}{(1+y)\log{\theta(x_1)}}\right).
\end{align}
We want to find some constant $C$ such that the following is true:
\begin{equation}
    1+\frac{y}{(1+y)\log{\theta(x)}} \geq \exp{\frac{Cx_2}{x_1\log{x_1}}} \label{eq317}
\end{equation}
\begin{equation}
    \log{\left(1+\frac{y}{(1+y)\log{\theta(x_1)}}\right)} \geq \frac{Cx_2}{x_1\log{x_1}}.
\end{equation}
Using again the fact that $\log{(1+x)} \geq \frac{x}{1+x}$, we want 
\begin{equation}
    \frac{\frac{y}{(1+y)\log{\theta(x_1)}}}{1+\frac{y}{(1+y)\log{\theta(x_1)}}} \geq \frac{Cx_2}{x_1\log{x_1}}
\end{equation}
\begin{equation} \label{eq20}
    \frac{y}{y + (1+y)\log{\theta(x_1)}} \geq \frac{Cx_2}{x_1\log{x_1}}.
\end{equation}
Multiplying the numerator and denominator of the LHS of \eqref{eq20} by $\theta(x_1)$ and using the definition of $y$, we have 
\begin{gather*}
    \frac{\theta(x_1)y}{\theta(x_1)y + (\theta(x_1)+\theta(x_1)y)\log{\theta(x_1)}}\\
 = \frac{0.998x_2-\log{2}}{0.998x_2 - \log{2}+(0.998x_2-\log{2} + \theta(x_1))(\log{\theta(x_1)})} \\
    \geq \frac{0.996x_2}{0.998x_2 - \log{2}+(0.998x_2-\log{2} + \theta(x_1))(\log{\theta(x_1)})}. \label{eq22}
\end{gather*}
To prove inequality \eqref{eq20}, it now suffices to show that
\begin{equation}
0.998x_2 - \log{2}+   (0.998x_2-\log{2} + \theta(x_1))(\log{\theta(x_1)}) \leq \frac{0.996}{C}x_1\log{x_1}.
\end{equation}
Using the fact that $\theta(x_1) < 1.000081x_1$ \cite[p. 194]{13} and $x_2 < \sqrt{2x_1}$ (Lemma \ref{xkLemma2}), we want to show the following for some constant $C$:
\begin{equation}
 0.998\sqrt{2x_1} - \log{2}+   (0.998\sqrt{2x_1}-\log{2} + 1.000081x_1)(\log{1.000081x_1}) \leq \frac{0.996}{C}x_1\log{x_1}. \label{eq25}
\end{equation} 
Using a computer, we find that $C = 0.99154$ is the largest positive value that satisfies \eqref{eq25} for all $x_1 \geq 120409$. Substituting this into \eqref{eq317}, we get the desired result. \end{proof}

In the remainder of the paper, we call the inequality $$\frac{\sigma(n)}{n\log\log n } < \frac{e^\gamma}{2}$$ for odd $n$ the \textit{analogue to Robin's inequality}.
\begin{lemma} \label{CAnum}
Suppose an odd $n$ violates the analogue to Robin's inequality. If $N\leq n \leq N^\prime$, where $N$ and $N^\prime$ are consecutive odd colossally abundant numbers, then $N$ or $N^\prime$ violates the analogue to Robin's inequality.
\end{lemma}
\begin{proof}
The proof is the same as that of \cite[Section 3, Proposition 1]{13}. See also the proof of Lemma \ref{lm44}.
\end{proof}

\begin{prop} \label{robinProp}
Assuming the Riemann hypothesis, we have \begin{equation}\label{RobinOddRe}
\frac{\sigma(n)}{n \log{\log{n}}}  <  \frac{e^\gamma}{2}
\end{equation} 
for all odd $n \geq c_0 = 18565284664427130919514350125 = 3^4 \cdot 5^3 \cdot 7^2 \cdot 11 \cdots 67$. 
\end{prop}
\begin{proof}
We restrict our attention to odd colossally abundant $n$. By Lemma \ref{CAnum}, if all odd colossally abundant numbers greater than or equal to $c_0$ satisfy $\frac{\sigma(n)}{n \log{\log{n}}} < \frac{e^\gamma}{2}$, then there can be no other values greater than $c_0$ that violate this inequality.

Using a computer, we find that the inequality (\ref{RobinOddRe}) is satisfied for all odd colossally abundant numbers with $61 < x_1 < 120409$ (see Section 6; we actually  calculated for $x_1\le 10^6$). The final odd colossally abundant number that violates this inequality for $x_1 < 120409$ is $N = 3^4\cdot5^3\cdot7^2\cdot11 \cdots 61 = 277093800961598968947975375$. The next odd colossally abundant number is $c_0$. It is necessary that we set the cutoff point at $c_0$  as Lemma \ref{CAnum} says that $n$ satisfies the analogue to Robin's inequality if the odd colossally abundant numbers on both sides of $n$ 
satisfy the analogue to Robin's inequality. We do not know whether the inequality is satisfied for $N<n<c_0$. In contrast to the situation in \cite{13}, the numbers here are 
large enough to make brute-force computation infeasible.

Hereinafter assume that $x_1\ge 120409$. 
Combining part (c) of Lemma \ref{xkLemma2} with Lemma \ref{logBound}, we have that, for $N$ odd colossally abundant,
\begin{equation} \label{eq327}
    \log{\log{N}} > \log{\theta(x_1)}\exp{\left(\frac{0.99154\sqrt{2}}{\sqrt{x_1}\log{x_1}} - \frac{0.486}{\sqrt{x_1}(\log{x_1})^2}\right)}.
\end{equation}
Let $S(x) = \theta(x) - x$. By \cite[p. 337]{17}, we have, assuming the Riemann hypothesis, that 
\begin{equation} \label{eq28}
    S(x)^2 \leq \frac{1}{64\pi^2}x(\log{x})^4 \text{ when } x\ge 599.
\end{equation}
As in \cite[p. 199]{13}, define the following:
$$\alpha(x) = \frac{S^2(x)(\log{x}+1.31)}{2x^{3/2}} + (c-2) + \frac{8-4c}{\log{x}} + \frac{2\log{x}}{x^{1/6}} + \frac{\log{2\pi}\log{x}}{x^{1/2}},$$
where $c = \gamma + 2 - \log{4\pi} = 0.04619$ by \cite[p. 194]{13}. Since $\alpha(x)$ is a decreasing function when $S(x)^2$ is replaced using \eqref{eq28}, we have $\alpha(x) \leq 2.6616$ for $x \geq 120409$.
By \cite[Lemma 5, p. 199]{13}, we have the following for $x \geq 20000$:
\begin{equation}
    \prod_{2 < p \leq x} \left(1-\frac{1}{p}\right)^{-1} \leq \frac{e^\gamma}{2}\log{\theta(x)}\exp{\left(\frac{2+c}{\sqrt{x}\log{x}} + \frac{\alpha(x)}{\sqrt{x}(\log{x})^2}\right)}. \label{eq29}
\end{equation}
By \cite[Lemma 6, p. 199]{13}, we have the following for $x \geq 20000$:
\begin{equation}
    \prod_{\sqrt{2x} < p \leq x} \left(1-\frac{1}{p^2}\right) \leq \exp{\left(\frac{-\sqrt{2}}{\sqrt{x}\log{x}} + \frac{4}{\sqrt{x}(\log{x})^2}\right)}. \label{eq30}
\end{equation}
By \cite[p. 204]{13}, we have that, for $N$ odd colossally abundant,
\begin{equation} \label{eq31}
    \frac{\sigma(N)}{N\log{\log{N}}} \leq \prod_{x_2 < p \leq x_1} \left(1-\frac{1}{p^2}\right) \prod_{2 < p \leq x_1} \left(1-\frac{1}{p}\right)^{-1}(\log{\log{N}})^{-1}.
\end{equation}
Combining \eqref{eq327}, \eqref{eq29}, \eqref{eq30}, \eqref{eq31}, we reach the following:
\begin{equation}
    \frac{\sigma(N)}{N\log{\log{N}}} \leq \frac{e^\gamma}{2}\exp{\left(\frac{-0.7702}{\sqrt{x_1}\log{x_1}} + \frac{7.1476}{\sqrt{x_1}(\log{x_1})^2}\right)}
\end{equation}
when $x_1\ge 120409$. 
Hence, we want to find when
\begin{equation}
    \frac{-0.7702}{\sqrt{x_1}\log{x_1}} + \frac{7.1476}{\sqrt{x_1}(\log{x_1})^2} \leq 0.
\end{equation}
This is true for $x_1 \geq 10724$. Since we already assume that $x_1 \ge 120409$, we have proved the desired result.
\end{proof}

\begin{theorem}{\rm (= Theorem \ref{MainThm1})}
The Riemann hypothesis is equivalent to the statement that 
\begin{equation}
\frac{\sigma(n)}{n\log{\log{n}}} < \frac{e^\gamma}{2}
\end{equation}
for all odd $n \ge c_0$, where $c_0$ is as in Proposition \ref{robinProp}.
\end{theorem}

\begin{proof}
Assuming the Riemann hypothesis is true, we have proved Proposition \ref{robinProp} above.

Suppose $\frac{\sigma(n)}{n \log{\log{n}}} \leq \frac{e^\gamma}{2}$ is true for odd $n \ge c_0$ and the Riemann hypothesis is false. Thus, by \cite{13}, $f(M) > e^\gamma$ for infinitely many even colossally abundant $M$. Since all colossally abundant numbers are even, we have that $M = 2^kN$, for $k > 0$ and $N$ odd. By the formula \eqref{expCA} for generating colossally abundant numbers, 
we have that if $M \rightarrow \infty$, then $N \rightarrow \infty$. Thus $N \ge c_0$ for sufficiently large $M$. By assumption, for odd $N \geq c_0$, we have $\frac{\sigma(N)}{N \log{\log{N}}} < \frac{e^\gamma}{2}$. But  
$$\frac{\sigma(2^kN)}{2^kN \log{\log{2^kN}}} \le \frac{\sigma(2^k)}{2^k} \frac{\sigma(N)}{N \log{\log{N}}} < 2\left(\frac{e^\gamma}{2}\right) = e^\gamma$$ for $N \geq c_0$. Thus Robin's inequality (\ref{Robineq}) is satisfied for sufficiently large colossally abundant $M$. This contradicts the statement that infinitely many such $M$ violate Robin's inequality. 
\end{proof}

\section{Lagarias's Criterion}
In this section we seek to establish, in the style of \cite{10}, an inequality equivalent to the Riemann hypothesis that is true for all odd integers.

Let $H_n$ be the harmonic series. Recall our definition of the new harmonic series $H_n^\prime$:
\begin{equation}
H_n^\prime = 2\sum_{k=1}^n \frac{1}{k} - \sum_{k=1}^{2n} \frac{1}{k} = 2H_n-H_{2n}. \label{hprime}
\end{equation}
We establish the following lemma as an analogue to \cite[Lemma 3.1]{10}.
\begin{lemma} \label{lm41}
For $n \geq 3$, $\frac{0.12n}{\log{n}}+\exp{H_n^\prime}\log{H_n^\prime} \geq \frac{e^\gamma}{2}n\log{\log{n}}$.
\end{lemma}
\begin{proof}
Equation 3.5 in \cite[p. 540]{10} states that
$$H_n = \log{n} + \gamma + \int_n^\infty \frac{\{t\}}{t^2} dt,$$
where $\{t\}$ is the fractional part of $t$. We can thus use \eqref{hprime} to rewrite $H_n^\prime$ as
\begin{align}
    H_n^\prime &= 2\left(\log{n} + \gamma + \int_n^\infty \frac{\{t\}}{t^2} dt\right) - \left(\log{2n} + \gamma + \int_{2n}^\infty \frac{\{t\}}{t^2} dt\right) \\
    &= \log{n} + \gamma - \log{2} +  \int_{n}^\infty \frac{\{t\}}{t^2} dt + \int_{n}^{2n} \frac{\{t\}}{t^2} dt. \label{eq36}
\end{align}
From \eqref{eq36}, we obtain the inequality
\begin{equation}
    H_n^\prime \geq \log{n}+\gamma-\log{2}.
\end{equation}
With this inequality, it suffices to show the following for suitable $C$:
\begin{equation}
    \frac{Cn}{\log{n}}+\frac{e^\gamma}{2}n\log{(\log{n}+\gamma-\log{2})} \geq \frac{e^\gamma}{2}n\log{\log{n}}.
\end{equation}
Let $t = \log{n}$. Equivalently, we want the following:
\begin{equation}
    \frac{C}{t}+\frac{e^\gamma}{2}\log{\left(1+\frac{\gamma - \log{2}}{t}\right)} \geq 0. \label{eq39}
\end{equation}
Let $g(t)$ be the LHS of \eqref{eq39}. For $t \geq 1$,
$$g^\prime(t) = \frac{-C}{t^2} + \frac{e^\gamma}{2} \left(\frac{\log{2}-\gamma}{t^2-t(\log{2}-\gamma)}\right) \leq 0$$
when
\begin{equation} \label{eqK}
C \geq \frac{0.104}{1-\frac{0.116}{t}}. \end{equation}
The RHS of \eqref{eqK} is decreasing for $t \geq 1$, so we have $g^\prime(t) \leq 0$ for $C = 0.12$. 
Furthermore, since $g(t)$ is decreasing and 
$$\lim_{t\to\infty} g(t) = 0,$$ 
we find that \eqref{eq39} must be true, and thus we have proven the desired result.
\end{proof}

\begin{lemma} \label{HnBound}
$0 < H^\prime_n - (\log{n}+\gamma-\log{2}) < \frac{3}{4n}$ for all $n\ge 1$.
\end{lemma}
\begin{proof}
We begin by demonstrating, for integers $a>b > 0$, that
\begin{equation}
    \int_a^b \frac{\{t\}}{t^2} dt < \frac{1}{2}\int_a^b \frac{1}{t^2} dt.
\end{equation}
Consider $0<k \in \mathbb{Z}$. We show the more restrictive inequality
\begin{equation}
    \int_k^{k+1} \frac{\{t\}}{t^2} dt < \frac{1}{2}\int_k^{k+1} \frac{1}{t^2} dt.
\end{equation}
Thus, our goal is to show that
\begin{equation}
    \int_0^1 \frac{t}{(t+k)^2} dt < \frac{1}{2}\int_k^{k+1} \frac{1}{t^2} dt.
\end{equation}
Integrating and rearranging terms, we want to show the following:
$$\log{(1+k)}-\log{k}-\frac{1}{k+1} < \frac{1}{2}\left(\frac{1}{k} - \frac{1}{k+1}\right)$$
\begin{equation}
    \frac{1}{2k} + \frac{1}{2(k+1)} - \log{\left(1+\frac{1}{k}\right)} > 0.
\end{equation}
Let $h(k)$ be the LHS of (4.11). We have the following two properties of $h(k)$:
$$\lim_{k\to\infty} h(k) = 0$$
$$h^\prime(k) = \frac{-1}{2k^2(k+1)^2} < 0.$$
By these two properties, we find that (4.11) must be true, which implies that (4.8) is true. Hence, we can assert the following about the error terms from (4.3):
\begin{equation}
    \int_{n}^\infty \frac{\{t\}}{t^2} dt + \int_{n}^{2n} \frac{\{t\}}{t^2} dt < \frac{1}{2}\left(\int_{n}^\infty \frac{1}{t^2} dt + \int_{n}^{2n} \frac{1}{t^2} dt\right)
    = \frac{3}{4n}.
\end{equation}
We note that by the same logic we can assert that the maximum error for approximating the harmonic series $H_n \approx \log{n}+\gamma$ is bounded by $\frac{1}{2n}$.
\end{proof}

\noindent The following lemma is an analogue to \cite[Lemma 3.2]{10}.
\begin{lemma} \label{lm43}
For $n \geq 3$, $\exp{H_n^\prime}\log{H_n^\prime} \leq \frac{e^\gamma}{2}n\log{\log{n}} + \frac{0.3n}{\log{n}}.$
\end{lemma}
\begin{proof}
By Lemma \ref{HnBound}, we have that
\begin{equation}
    H^\prime_n < \log{n} + \gamma - \log{2} + \frac{3}{4n}.
\end{equation}
For $n \geq 7$, $\gamma - \log{2} + \frac{3}{4n} < 0$. Thus we have that, for $n \geq 7$, 
\begin{equation}
    H^\prime_n < \log{n}.
\end{equation}
Furthermore, for $n \geq 3$, we have that
\begin{equation}
    e^{\frac{3}{4n}} < 1+\frac{1}{n}.
\end{equation}
Combining (4.13), (4.14), and (4.15), we have the following for $n \geq 7$:
\begin{equation}
    \exp{H^\prime_n}\log{H^\prime_n} < \frac{e^\gamma}{2}n\exp{\left(\frac{3}{4n}\right)}\log{\log{n}} < \frac{e^\gamma}{2}n\log{\log{n}} + \frac{e^\gamma}{2}\log{\log{n}}.
\end{equation}
Since, for $n \geq 3$,
\begin{equation}
    \frac{e^\gamma}{2}\log{\log{n}} < \frac{0.3n}{\log{n}},
\end{equation}
we have that 
\begin{equation}
    \exp{H_n^\prime}\log{H_n^\prime} \leq \frac{e^\gamma}{2}n\log{\log{n}} + \frac{0.3n}{\log{n}}
\end{equation}
for $n \geq 7$. The result follows directly from checking the cases $3 \leq n \leq 7$.
\end{proof}

We need an analogue of Lemma \ref{CAnum} in order to do the computations needed to establish an analogue of Lagarias's criterion for the Riemann hypothesis.
\begin{lemma} \label{lm44}
Suppose the Riemann hypothesis is true. Fix $k > 0$. Suppose $n$ is odd and 
\begin{equation}\label{RobIneq2}
\sigma(n) \geq \frac{3n}{\log{n}} + \frac{e^\gamma}{2}n\log\log{n}.
\end{equation} 
If $N_1\leq n \leq N_2$ and $N_1 \geq 490$, where $N_1$ and $N_2$ are consecutive odd colossally abundant numbers, then $N_1$ or $N_2$ satisfies the inequality (\ref{RobIneq2}).
\end{lemma}
\begin{proof}
Our proof is based on that of \cite[Section 3, Prop. 1]{13}, but involves more complicated calculations.
By the definition of colossally abundant, there exists a value of $\epsilon$ such that 
\begin{equation} \label{eq4-20}
    \frac{\sigma(n)}{n^{\epsilon +1}} \leq \frac{\sigma(N_1)}{N_1^{\epsilon +1}} = \frac{\sigma(N_2)}{N_2^{\epsilon +1}}
\end{equation}
(this is one of the countably many values of $\epsilon$ that are excluded in some treatments of colossally abundant numbers).
We want to show the following for $N$ equal to one of $N_1$ or $N_2$:
\begin{equation} \label{eq4-21}
    \frac{\sigma(N)}{\frac{3N}{\log{N}} + \frac{e^\gamma}{2}N\log\log{N}} \geq \frac{\sigma(n)}{\frac{3n}{\log{n}} + \frac{e^\gamma}{2}n\log\log{n}} \geq 1.
\end{equation}
The second inequality in \eqref{eq4-21} is true by assumption. To prove the first inequality, we want to show that
\begin{equation} \label{eq4-22}
    \frac{\sigma(n)}{\sigma(N)} \leq
    \frac{n^{\epsilon+1}}{N^{\epsilon+1}} \leq
    \frac{\frac{3n}{\log{n}} + \frac{e^\gamma}{2}n\log\log{n}}{\frac{3N}{\log{N}} + \frac{e^\gamma}{2}N\log\log{N}}.
\end{equation}
The first inequality in \eqref{eq4-22} is easily deduced from \eqref{eq4-20}. We can rewrite the second inequality as 
\begin{equation}
    \frac{\frac{3N}{\log{N}} + \frac{e^\gamma}{2}N\log\log{N}}{N^{\epsilon+1}} \leq \frac{\frac{3n}{\log{n}} + \frac{e^\gamma}{2}n\log\log{n}}{n^{\epsilon+1}}.
\end{equation}
Making the substitution $t_0 = \log{N}$, $t_1 = \log{n}$, and taking the $\log$ of both sides, we can rewrite the inequality in the more convenient form
\begin{equation} \label{eq424}
    \log{\left(\frac{3}{t_0}+\frac{e^\gamma}{2}\log{t_0}\right)} - \epsilon t_0 \leq \log{\left(\frac{3}{t_1}+\frac{e^\gamma}{2}\log{t_1}\right)} - \epsilon t_1.
\end{equation}
The function 
\begin{equation}
    s(t) = \log{\left(\frac{3}{t}+\frac{e^\gamma}{2}\log{t}\right)} - \epsilon t
\end{equation}
is concave down for $t > 6.193$. Since $N$ is equal to one of $N_1$ or $N_2$, the inequality \eqref{eq424} is true for $N_1 \geq e^{6.193}$, which is $N_1 \geq 490$.
\end{proof}

We now establish an analogue to \cite[Theorem 3.2]{10}.

\begin{lemma} \label{lm45}
If the Riemann hypothesis is false, then there exist constants $0 < \beta < 1/2$ and $C > 0$ such that $$\sigma(N) \geq \frac{e^\gamma}{2} N \log{\log{N}} + \frac{CN\log{\log{N}}}{(\log{N})^\beta}$$ holds for infinitely many odd colossally abundant $N$.
\end{lemma}
\begin{proof}
Let $M$ run through the even colossally abundant numbers. By \cite[p. 205]{13},  we have $f(M) = e^\gamma(1 + \Omega_\pm(\log{M})^{-\beta})$, where $f$ is as in (\ref{Robineq}). Let $N$ be the odd colossally abundant number formed by removing all factors of 2 from a colossally abundant number $M$, so $M = 2^kN$. We have 
\begin{equation}
    f(M) = \frac{2^{k+1}-1}{2^k}\frac{\sigma(N)}{N \log\log{2^kN}} < 2f(N).
\end{equation}
By Lemma \ref{expbound},
$$
\log M = \log(2^kN) = \log N + k\log 2 = \log N + O(\log\log N) .
$$
It follows that 
$$
f(M)= e^{\gamma}(1+\Omega_{\pm}(\log N)^{-\beta}).
$$
Thus,  for some $C > 0$, $f(N) > \frac12 f(M) \geq \frac12 e^\gamma+\frac{C}{(\log N)^\beta}$ for infinitely many odd colossally abundant $N$. The lemma follows directly.
\end{proof}

\begin{theorem} \label{lagariasThm}{\rm (= Theorem \ref{MainThm2})}
The Riemann hypothesis is equivalent to the statement that
\begin{equation} \label{thm41equ}
\sigma(n) \leq \frac{3n}{\log{n}} + \exp{H^\prime_n}\log{H^\prime_n}
\end{equation} for all odd $n \geq 3$.
\end{theorem}
\begin{proof}
We imitate the proof of \cite[Theorem 1.1]{10}. Suppose the Riemann hypothesis is true. Then Theorem \ref{MainThm1} and Lemma 4.1 together give, for $n \ge c_0$,
$$\sigma(n) < \frac{e^\gamma}{2}n\log{\log{n}} \leq \frac{.12n}{\log{n}} + \exp{H_n^\prime}\log{H_n^\prime}\leq \frac{3n}{\log{n}} + \exp{H_n^\prime}\log{H_n^\prime}.$$
It is computationally infeasible to test $\sigma(n) \leq \frac{3n}{\log{n}} + \exp{H^\prime_n}\log{H^\prime_n}$ for $n$ up to $c_0$, as $c_0$ is on the order of $10^{27}$. For $n \leq c_0$, we aim to prove \eqref{thm41equ} by demonstrating the following sequence of inequalities for $3 \leq n \leq c_0$:
\begin{equation} \label{ineq25}
\sigma(n) \leq \frac{e^\gamma}{2}n\log{\log{n}} + \frac{2.8n}{\log{n}} \leq \frac{3n}{\log{n}} + \exp{H_n^\prime}\log{H_n^\prime}. \end{equation}
The second inequality in \eqref{ineq25} is a result of Lemma 4.1. A computer calculation using the list of all 24 odd colossally abundant numbers up to $c_0$ (see Section 6) shows that 
there are no odd colossally abundant violations to $\sigma(n) \leq \frac{e^\gamma}{2}n\log{\log{n}} + \frac{2.8n}{\log{n}}$ for $n \leq c_0$.  Thus, this inequality is satisfied for all $n \leq c_0$ by Lemma 4.4. 

Suppose \eqref{thm41equ} holds for all odd $n \geq 3$ and the Riemann hypothesis is false. Combining Lemmas \ref{lm43} and \ref{lm45}, we have the following
for infinitely many $n$:
\begin{equation}
    \frac{e^\gamma}{2} n \log{\log{n}} + \frac{Cn\log{\log{n}}}{(\log{n})^\beta} \leq \sigma(n) \leq \frac{3n}{\log{n}} + \exp{H^\prime_n}\log{H^\prime_n} \leq  \frac{e^\gamma}{2}n\log{\log{n}} + \frac{3.3n}{\log{n}}.
\end{equation}
This implies that
\begin{equation} \label{eq27}
    C\log{\log{n}} < 3.3(\log{n})^{\beta-1}. 
\end{equation}
Since $\beta < \frac{1}{2}$, we have that \eqref{eq27} is a contradiction for sufficiently large $n$, as the LHS tends towards infinity while the RHS tends toward 0. Thus the Riemann hypothesis must be true if \eqref{thm41equ} holds for all odd $n \geq 3$.
\end{proof}

\begin{remark}
Note that the value of 2.8 in \eqref{ineq25} was simply needed to account for the small values of $n$. In fact, it may be possible to improve the bound on $\sigma(n)$ significantly for larger values of $n$.
\end{remark}

\section{Robin's Inequality for Various Sets of Integers}
In this section we prove Robin's inequality (\ref{Robineq}) and its analogue over certain sets of integers. We construct analogues of the work of \cite{6} with inequalities concerning squarefree numbers. We also prove that an integer satisfies Robin's inequality if the $p$-adic valuation of $n$ is sufficiently large based on $n$.

\subsection{Squarefree Numbers}

\begin{theorem} \label{thm51}
Odd squarefree numbers $n\ge 4849845$ satisfy $\frac{\sigma(n)}{n \log{\log{n}}} < \frac{e^\gamma}{2}$.
\end{theorem}

\begin{remark}
The smallest odd integer with seven distinct prime factors is $n = 4849845 = 3 \cdot 5 \cdot 7 \cdot 11 \cdot 13 \cdot 17 \cdot 19$.
\end{remark}

\begin{proof}
Let $m$ be the number of prime factors of an odd squarefree $n$. First, we show that squarefree integers $n \geq 4849845$ with six or fewer prime factors ($m \leq 6$) satisfy $\frac{\sigma(n)}{n \log{\log{n}}} < \frac{e^\gamma}{2}$. We note the following for constant $m$: 
\begin{itemize}
    \item Increasing the value of a prime factor in $n$ to another prime factor that does not divide $n$ decreases the value of $f(n)=\sigma(n)/n \log\log n$. The only possible way to increase both $n$ and $f(n)$ simultaneously is to increase one prime factor in $n$ and decrease another prime factor in $n$.
    \item Consider all odd squarefree numbers with $m$ prime factors and largest prime factor $q$. The function $f(n)$ is maximized over these integers if $n$ is the product of $q$ and the first $m-1$ primes.
    \item If we increase $q$, our maximal value of $f(n)$ decreases.
\end{itemize}

We proceed by analyzing each case $m = 1$ to $m = 6$.

\boldmath $m = 1$: \unboldmath $p = 29 < 4849845$ is the first prime that satisfies the analogue to Robin's inequality in the statement of the theorem. Therefore, all primes 
$p\ge 29$ satisfy the inequality.

\boldmath $m = 2$: \unboldmath $n = 3 \cdot 37$ satisfies the analogue to Robin's inequality, and therefore all $n=pq$ with primes $p<q$ and $q\ge 37$ satisfy the inequality. Since $29\cdot 31 < 4849845$, the theorem is true for all odd squarefree numbers with 2 prime factors.

\boldmath $m = 3$: \unboldmath $n = 3 \cdot 5 \cdot 41$ satisfies the analogue to Robin's inequality. Since $31\cdot 37\cdot 41 < 4849845$, the theorem is true for $m=3$.

\boldmath $m = 4$: \unboldmath $n = 3 \cdot 5 \cdot 7 \cdot 37$ satisfies the analogue to Robin's inequality. Since $23\cdot 29\cdot 31\cdot 37 < 4849845$, the theorem is true 
for $m=4$.

\boldmath $m = 5$: \unboldmath $n = 3 \cdot 5 \cdot 7 \cdot 11 \cdot 29$ satisfies the analogue to Robin's inequality. Since $13\cdot 17\cdot 19\cdot 23\cdot 29 < 4849845$, 
the theorem is true for $m=5$.

\boldmath $m = 6$: \unboldmath $n = 3 \cdot 5 \cdot 7 \cdot 11 \cdot 13 \cdot 23$ satisfies the analogue to Robin's inequality, so the inequality holds when $m=6$ and the largest prime factor is at least 23.   Moreover,    a calculation shows that all odd squarefree numbers $n \geq 4849845$ with $m=6$ and $q \leq 23$ satisfy the inequality.

For values of $n$ with seven or more prime factors, we proceed in the manner of \cite[pp. 360--361]{6}. It suffices to consider $n$ of the form $\prod_{n=1}^m q_j$, where $q_1=3, q_2=5, \dots$,
and $q_m$ is the $m$th odd prime. 

\noindent \textbf{Case 1:} $q_m \geq \log{(q_1 \dots q_m)} = \log{n}$ \\
The proof of this case is identical to that of \cite[p. 360]{6}.

\noindent \textbf{Case 2:} $q_m < \log{(q_1 \dots q_m)} = \log{n}$ \\
As in \cite{6}, we note that 
\begin{equation} \label{eq51}
    \sum_{j=1}^{m} (\log{(q_j+1)} - \log{q_j}) = \sum_{j=1}^{m} \int_{q_j}^{q_j+1} \frac{dt}{t} < \sum_{j=1}^{m} \frac{1}{q_j}.
\end{equation}
As the LHS of \eqref{eq51} equals $\log{\left( \frac{\sigma(n)}{n} \right)}$, we want to show that, when $m\ge 7$, 
\begin{equation} \label{eq52}
    \sum_{j=1}^{m} \frac{1}{q_j} < \gamma - \log2+ \log{\log{\log{n}}}.
\end{equation}
By \cite[Corollary (3.20)]{15}, we have that (with $p = 2$ included)
\begin{equation}
    \sum_{p \leq x} \frac{1}{p} < \log{\log{x}} + B + \frac{1}{\log^2{x}},
\end{equation}
where $B \approx 0.2615$ is the Meissel-Mertens constant. Since $m\ge 7$, we have $q_m\ge 19$, which implies that
\begin{equation} \label{eq54}
    \log{\log{q_m}} + B - \frac{1}{2} + \frac{1}{\log^2 q_m} < \gamma - \log2 + \log{\log{q_m}}.
\end{equation}
Combining (5.3), (5.4), and the assumption that $q_m < \log{n}$, we have that, for odd primes $p$,
\begin{equation}
    \sum_{j = 1}^m \frac{1}{q_j} < \gamma - \log2 + \log{\log{q_m}} < \gamma - \log2 + \log{\log{\log{n}}},
\end{equation}
as desired.
\end{proof}
\begin{remark}
A calculation plus Lemma \ref{lm41} and Theorem \ref{thm51} shows that the inequality \eqref{thm41equ} in Theorem \ref{lagariasThm} is true for all odd squarefree numbers.
\end{remark}

Note that our proof of Theorem 5.1 is almost identical to that for squarefree numbers in \cite[pp. 360--361]{6}; the only difference is the use of a sharper bound for the sum of reciprocals of primes in (5.3). Furthermore, by replacing (5.3) with the even sharper estimate of \cite[Theorem 5]{15}, which states that, for $x \geq 286$, 
\begin{equation} \label{eq56}
    \sum_{p \leq x} \frac{1}{p} < \log{\log{x}} + B + \frac{1}{2\log^2{x}},
\end{equation}
we can prove the following:
\begin{corollary} \label{largeSquarefree}
    Odd squarefree numbers with at least 13 prime factors satisfy $\frac{\sigma(n)}{n \log{\log{n}}} < 0.45e^\gamma$.
\end{corollary}
\begin{proof}
    Analogous to \eqref{eq52}, we want to show that if $q_m < \log{n}$ then
    \begin{equation} \label{eq57}
    \sum_{j=1}^{m} \frac{1}{q_j} < \gamma + \log 0.45 + \log{\log{\log{n}}}.
    \end{equation}
    Furthermore, analogous to \eqref{eq54}, we have the following: if $q_m \geq 286$ then
    \begin{equation} \label{eq58}
        \log{\log{q_m}} + B - \frac{1}{2} + \frac{1}{2\log^2 q_m} < \gamma + \log 0.45 + \log{\log{q_m}}.
    \end{equation}
    Combining \eqref{eq56} (without $p = 2$) with \eqref{eq58}, we reach \eqref{eq57}. A computer calculation for $n=\prod_{i=1}^{m} q_i$ for
$43 \leq q_m \leq 286$ yields the desired result for these $n$,
from which the corollary follows.
\end{proof}

\begin{prop}
As $n$ runs through odd squarefree integers,  $\limsup \frac{\sigma(n)}{n \log{\log{n}}}=\frac{4e^\gamma}{\pi^2}$.
\end{prop}
\begin{proof}
It suffices to consider $n$ of the form $q_1q_2\cdots q_m$. 
We have that
$$\frac{\sigma(n)}{n} = \prod_{1 \leq i \leq m} \left(\frac{1+q_i}{q_i} \right)$$
and 
$$\log{q_1q_2 \dots q_n} = \sum_{1\leq i \leq m} \log{q_i} = \theta(q_m) - \log{2} \approx q_m$$
since $\lim_{n \rightarrow \infty} \frac{\theta(n)}{n} = 1$. Combining these two expressions, we have that
\begin{equation} \label{robinAsy}
\frac{\sigma(n)}{n \log\log{n}} \approx \frac{\prod_{1 \leq i \leq m} \left(\frac{1+q_i}{q_i} \right)}{\log{q_m}}.
\end{equation}
A theorem of Mertens implies that 
$$\prod_{3\leq p \leq x} \left(1-\frac{1}{p} \right)^{-1} \approx \frac{e^\gamma}{2}\log{x}.$$
Since $$\prod_{p \geq 2} \left(1-\frac{1}{p^2} \right)^{-1} = \sum_{n=1}^{\infty} \frac{1}{n^2} = \frac{\pi^2}{6},$$
we have, for odd primes $p$, $$\prod_{p \geq 3} \left(1-\frac{1}{p^2} \right)^{-1} = \frac{\pi^2}{8}.$$
Thus, for odd primes $q$,
$$\prod_{q \leq q_m} \left(\frac{1+q}{q}\right) \approx  \frac{8}{\pi^2}\prod_{q \leq q_m} \left(1-\frac{1}{q} \right)^{-1} \approx \frac{8}{\pi^2}\frac{e^\gamma}{2} \log{q_m}.$$
Substituting this into \eqref{robinAsy}, we have that
$$\frac{\sigma(n)}{n \log\log{n}} \approx \frac{4e^\gamma}{\pi^2}.$$
\end{proof}
\noindent Thus we can see that the constant 0.45 in Corollary \ref{largeSquarefree} can be improved only slightly to $\frac{4e^\gamma}{\pi^2} \approx 0.4053$.

In the vein of Theorem \ref{MainThm1}, we extend \cite[Theorem 1.1]{6}, which states that all squarefree integers greater than 30 satisfy Robin's inequality, to the following: \begin{theorem} Let $n=2^km$, where $m$ is an odd, squarefree integer and $k\ge 0$. If $n\ge 841$ then
Robin's inequality (\ref{Robineq})  is true for $n$.
\end{theorem}
\begin{proof}Let 
$$ 
C_1=\frac{\sigma(m)}{m \log\log m}
$$
and let $C_2=2C_1/e^{\gamma}$. If $2^k$ is a power of 2 with 
\begin{equation}\label{2^kinequality}
2^k\ge (1/m)e^{(\log m)^{C_2}}
\end{equation}
then $\log\log(2^k m) \ge C_2\log\log m$, so
$$
\frac{\sigma(2^km)}{2^km\log\log(2^km)} < \frac{2\sigma(m)}{m\log\log(2^km)} \le \frac{2\sigma(m)}{C_2m\log\log m} = e^{\gamma}.
$$

 If $m\ge 4849845$, then Theorem 5.1 says that $C_1<e^{\gamma}/2$, so $C_2<1$. Therefore, every $k\ge 0$ satisfies (\ref{2^kinequality}), so we obtain Robin's inequality for $2^km$.

For each squarefree $m<4849845$, there is a small finite set of values of $k$ satisfying (\ref{2^kinequality}), and a straightforward computer computation verifies Robin's inequality
with a finite number of exceptions, the largest being $n=840$. \end{proof}

\cite[Theorem 1.2]{6} states that any odd number greater than 9 satisfies Robin's inequality. This result can be improved as well. 
Eum and Koo, without assuming the Riemann hypothesis, prove the following theorem using methods similar to those of \cite{6}. 
\begin{theorem}{\rm \cite{9}}
 For odd $n > 21$, $\frac{\sigma(n)}{n\log\log{n}} < \frac{3}{4}e^\gamma$.
\end{theorem}

\subsection{General Powers of Primes}
Recall that  $f(n) = \frac{\sigma(n)}{n\log{\log{n}}}$. Consider an integer $n$ and an integer $p^kn$ with $k > 0$ such that $p \nmid n$. We have the following:
\begin{align}
    \log{\log{p^kn}} &= \log{\left( k\log{p} + \log{n} \right)} \\
    &= \log{\left( (\log{n}) \left( 1+\frac{k\log{p}}{\log{n}} \right) \right)} \\
    &= (\log{\log{n}}) \left( 1 + \frac{\log{\left( 1+\frac{k\log{p}}{\log{n}} \right)}}{\log{\log{n}}} \right).
\end{align}
Since $\frac{\sigma(p^kn)}{p^kn} = \frac{\sigma(p^k)}{p^k} \frac{\sigma(n)}{n}$, we have 
\begin{align} \label{gDef}
    f(p^kn) &= \frac{\sigma(p^k)}{p^k}\left( 1 + \frac{\log{\left( 1+\frac{k\log{p}}{\log{n}} \right)}}{\log{\log{n}}}\right)^{-1} f(n) \\
&= g(n, k, p)f(n),
\end{align}
where the multivariable function $g(n, k, p)$ is the factor by which $f(n)$ changes when $n$ is multiplied by $p^k$ when $p \nmid n$. 

\begin{theorem} \label{generalK}
    Given an arbitrary prime $p$ that does not divide the positive integer $n$, and an arbitrary positive constant $c$, 
if $$k > \frac{\log{n}}{\log{p}} \left( (\log{n})^{\frac{p+c-cp}{c(p-1)}}-1 \right)$$
then $g(n, k, p) < c$.
\end{theorem}
\begin{proof}
    We want $g(n, k, p) < c$, so by the definition of $g(n, k, p)$ in \eqref{gDef}, we have
    
    \begin{equation}
        \frac{\sigma(p^k)}{p^k}\left( 1 + \frac{\log{\left( 1+\frac{k\log{p}}{\log{n}} \right)}}{\log{\log{n}}}\right)^{-1} < c
    \end{equation}
    \begin{equation}
        \iff \frac{p^{k+1}-1}{p^{k+1}-p^k} < c\left( 1 + \frac{\log{\left( 1+\frac{k\log{p}}{\log{n}} \right)}}{\log{\log{n}}}\right)
    \end{equation}
    \begin{equation}
        \iff (\log{\log{n}})\left(\frac{p^{k+1}(1-c)+cp^k-1}{p^{k+1}-p^k}\right) < c\log{\left(1 + \frac{k\log{p}}{\log{n}} \right)}.
    \end{equation}
    Exponentiating both sides, we find that we want
    \begin{equation} \label{eq16}
        (\log{n})^{\left(\frac{p^{k+1}(1-c)+cp^k-1}{p^{k+1}-p^k}\right)} < \left(1 + \frac{k\log{p}}{\log{n}} \right)^c.
    \end{equation}
    For constant $c$, we have that $\frac{p^{k+1}(1-c)+cp^k-1}{p^{k+1}-p^k}$ is increasing as $k$ increases and that $$\lim_{k\to\infty} \frac{p^{k+1}(1-c)+cp^k-1}{p^{k+1}-p^k} = \frac{p+c-cp}{p-1}.$$ Thus, we know that (5.13) is true if 
    \begin{equation} \label{eq17}
        (\log{n})^{\left( \frac{p+c-cp}{p-1} \right)} < \left(1 + \frac{k\log{p}}{\log{n}} \right)^c.
    \end{equation}
    Solving for $k$, we find that $g(n, k, p) < c$ if 
    \begin{equation} \label{kBound}
        k > \frac{\log{n}}{\log{p}} \left( (\log{n})^{\frac{p+c-cp}{c(p-1)}}-1\right),
    \end{equation}
    as desired.
\end{proof}

A number of interesting results can be deduced from Theorem \ref{generalK}. Setting $c = 1$, we obtain the following:
\begin{corollary} \label{c1k}
Let $n$ be an integer and $p$ be a prime that does not divide $n$. Then $f(p^kn) < f(n)$ if $$k > \frac{\log{n}}{\log{p}} \left( (\log{n})^{\frac{1}{p-1}}-1\right).$$ 
\end{corollary}
Thus by Corollary \ref{c1k}, we conclude that if $n$ satisfies Robin's inequality  (\ref{Robineq}),  then $p^kn$ must satisfy Robin's inequality for sufficiently large $k$. Applying this conclusion to odd numbers, we know from \cite{6} that all odd numbers greater than 9 satisfy Robin's inequality (\ref{Robineq}). By Corollary \ref{c1k} we note that if 
\begin{equation} \label{eq19}
k > \frac{\log{n}}{\log{2}} \left( (\log{n})-1\right) \end{equation}
then $2^kn$ must satisfy Robin's inequality as well. However, it should be noted that since the bound for $k$ is $O(\log^2{n})$, this bound fails to encompass the colossally abundant numbers by far. For example, the odd colossally number 135135 uses $k > 184.3$ according to \eqref{eq19}.

Theorem \ref{generalK} allows us to create a class of integers that satisfy Robin's inequality and resemble the colossally abundant numbers. We call these numbers \textit{colossally abundant-like} numbers. The algorithm for constructing a \textit{colossally abundant-like} number $N$ is as follows:
\begin{enumerate}
    \item Select a largest prime factor $x = p_n$, the $n$th prime. Let $N = p_n$. 
    \item Set $p = p_{n-1}$ and choose $c > 0$.
    \item Compute the lower bound of $k$ with \eqref{kBound}. Let $L = \lceil k \rceil$.
    \item Multiply $N$ by $p_{n-1}^{L}$. Our new $N = p_np_{n-1}^{L}$.
    \item Repeat steps 2 to 4 with successively decreasing primes. Note that our choice of $c$ may change after each iteration.
\end{enumerate}
Let $c_i$ be the constant $c$ we choose on the $i$th iteration of steps 2 through 4 in the algorithm above. We have formed an integer that satisfies Robin's inequality if $\prod_{i}c_i < \frac{e^\gamma}{f(p_n)}$.

\begin{example} \label{ca-like}
We will construct a \textit{colossally abundant-like} number with largest prime factor $p_n = 17$. Thus we set $N = 17$.

\medskip
\noindent 
\begin{center}
\begin{tabular}{|c|c|c|c|}\hline
$p$ & $c$ & $k$ & $N$\\\hline
13 & 0.67 & 0.9954 & $17 \cdot 13$  \\\hline
11 & 0.91 & 0.9499 & $17 \cdot 13 \cdot 11$ \\\hline
7 & 1.06 & 0.91969 & $17 \cdot 13 \cdot 11 \cdot 7$ \\\hline
5 & 1.12 & 1.8306 & $17 \cdot 13 \cdot 11 \cdot 7 \cdot 5^2$ \\\hline
3 & 1.38 & 2.94396 & $17 \cdot 13 \cdot 11 \cdot 7 \cdot 5^2 \cdot 3^3$ \\\hline
2 & 1.75 & 11.4776 & $17 \cdot 13 \cdot 11 \cdot 7 \cdot 5^2 \cdot 3^3 \cdot 2^{12}$ \\\hline
\end{tabular}
\end{center}
\medskip

\noindent We note that $N = 17 \cdot 13 \cdot 11 \cdot 7 \cdot 5^2 \cdot 3^3 \cdot 2^{12}$ satisfies Robin's inequality (\ref{Robineq}), as 
$$\prod_i c_i = 1.748 < \frac{e^\gamma}{f(17)} = 1.751.$$
\end{example}

Unfortunately we cannot use this algorithm to construct the colossally abundant numbers and show that they satisfy Robin's inequality, even if we optimize our choices of $c_i$ to produce a $k$ as close to the desired exponent as possible. The main issue is in computing the exponents of the smaller primes - if we choose a small $c$, we require $L$ to be too large for $N$ to be colossally abundant. But if we choose a large $c$ in order to attain a smaller $L$, we lose the ability to show that $N$ satisfies Robin's inequality. The discrepancy between whether our algorithm is able to show that $N$ satisfies Robin's inequality and whether $N$ actually satisfies Robin's inequality lies in the step where we take a limit, between equations \eqref{eq16} and \eqref{eq17}.

\begin{example}
Note that the prime factorization of $N$ in Example \ref{ca-like} is almost identical to that of the colossally abundant number $367567200 = 17 \cdot 13 \cdot 11 \cdot 7 \cdot 5^2 \cdot 3^3 \cdot 2^{5}$, only differing in the 2-adic order. 
\end{example}

\subsection{Small Powers of Primes}
\cite{6} shows that all odd numbers $n$ greater than 9 satisfy Robin's inequality, and we have proven that all even numbers of the form $2^kn$, where $k$ is sufficiently large depending on $n$, satisfy Robin's inequality. What remains is to show, for all even numbers with 2-adic order below the bound in \eqref{eq19}, that Robin's inequality is satisfied. Unfortunately, we do not prove this assertion, though we will provide a number of unproven observations regarding $g(n, k, p)$.

As we are primarily concerned with the 2-adic order, we set $p = 2$. If $f(2^kn) = g(n, k, 2)f(n)$, we have that 
\begin{equation}
    g(n, k, 2) = \frac{2^{k+1}-1}{2^k}\left(1+\frac{\log\left(1+\frac{k\log 2}{\log n}\right)}{\log\log n}\right)^{-1}.
\end{equation}
Below is a graph of $g(10^5, k, 2)$.

\begin{center} 
\includegraphics[height=.8in]
{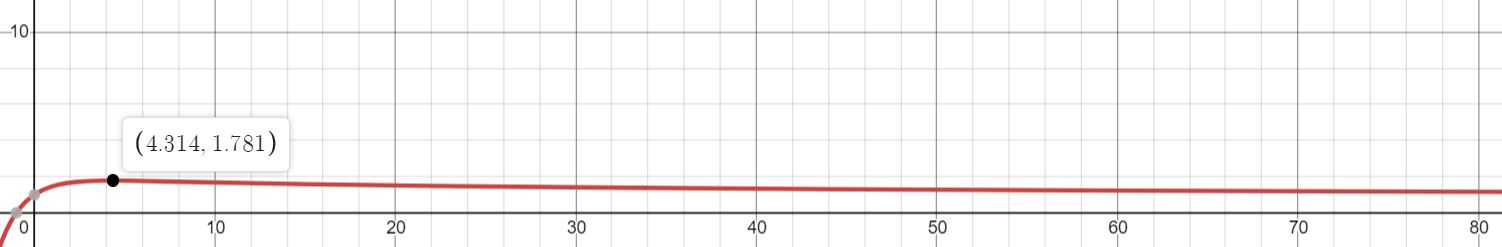}
\end{center}
The graphs for $n > 100$ are very similar to the one pictured. The value of $g(n, k, 2)$ peaks at a positive value, then decreases asymptotically to 0, albeit very slowly. Taking the partial derivative with respect to $k$ of $g(n, k, 2)$, we find that the maximum of $g(n, k, 2)$ for fixed $n$ occurs where $k$ satisfies
\begin{equation} \label{maxK}
    (\log{( 2^kn )}) (\log(\log 2^kn )) - 2^{k+1} + 1 = 0.
\end{equation}

For fixed odd $n$, $f(2^mn)$ is maximized by either $m = \lfloor k \rfloor$ or $m = \lceil k \rceil$ as given in \eqref{maxK}. Investigation into this concept reveals that the colossally abundant numbers do not necessarily maximize $f(2^mn)$ for fixed odd $n$. 

\begin{example}
Take the odd colossally abundant number $n = 135135$. We have that $k = 4.358$, and $f(2^5 \cdot 135135) = 1.72355$, while $f(2^4 \cdot 135135) = 1.72557$. But $2^5\cdot135135$ is the corresponding colossally abundant number, and it does not maximize $f(2^mn)$ for $n = 135135$.
\end{example}

\begin{prop}
As $n$ increases, the value of $k$ in \eqref{maxK} increases as well.
\end{prop}
\begin{proof}
Let $x = 2^k$. We rewrite \eqref{maxK} as follows
\begin{equation}
    (\log{xn})(\log\log{xn})-2x+1=0.
\end{equation}
Implicitly differentiating, we reach the following:
\begin{equation} \label{deriv23}
    \left(\frac{dx}{dn}\right) \left(2xn-n\log\log{xn}-n \right) = x\log\log{xn}+1.
\end{equation}
Substituting $\log\log{xn} = \frac{2x-1}{\log{xn}}$ into \eqref{deriv23}, we want to show that
\begin{equation}
    \frac{dx}{dn} = \frac{x\log\log{xn} + 1}{n\left(2x - \frac{2x-1}{\log{xn}} - 1 \right)} > 0.
\end{equation}
This is equivalent to
\begin{equation}
    2x - 1 - \frac{2x-1}{\log{xn}} = (2x-1)\left(1-\frac{1}{\log{xn}}\right) > 0,
\end{equation}
which is true because $\log{xn} > 1$.
\end{proof}

Experimental evidence allows us to conjecture that as $n$ increases, the maximal value of $g(n, k, 2)$ increases asymptotically toward 2.

Note that the formula for $g(n, k, 2)$ is composed completely of elementary functions; the value of $g(n, k, 2)$ is solely dependent on the magnitude of $n$ and $k$, and not on the prime factors of $n$. Thus, we can look for violations to Robin's inequality greater than 5040, if they exist, involving a large odd $n$ and a 2-adic order $k$ as given by (5.22).

\section{Computation of Odd Colossally Abundant Numbers}
In this section we present an algorithm to test the analogue  (\ref{RobinOdd})  of Robin's inequality on all odd colossally abundant numbers up to a bound. We use an algorithm similar to that of \cite{3}. This algorithm allows us to efficiently test Robin's inequality on large colossally abundant numbers.

First we generate a list of primes up to some bound $x$. We call this set \textbf{primeSet}.

Our task now is to generate the set of $\epsilon$ values for all odd colossally abundant numbers with largest prime factor less than $x$. We find the minimum $\epsilon$ in our set, which is the $\epsilon$ needed for the addition of the greatest prime less than $x$. We call this value \textbf{minEpsilon}. The formula is given in Lemma 2.2: $$\epsilon > \log_p{(p+1)}-1.$$

From Section 2, we see that each colossally abundant number can be defined by a range of $\epsilon$ values. By \cite[p. 253]{3}, when $\epsilon$ passes over the threshold value as defined below, the prime $p$ in the corresponding colossally abundant number has its exponent increased from $a-1$ to $a$:
\begin{equation} \label{eqA1}
    \epsilon = \log_p{\left( \frac{1-p^{a+1}}{p-p^{a+1}} \right)}.
\end{equation}
Our algorithm for generating the set of $\epsilon$ is below:

{\bf Result:} Generate all $\epsilon$ for all odd colossally abundant numbers with largest prime factor $\leq x$.

 Initialize an empty list of tuples that will contain epsilons and their corresponding prime. Call this list \textbf{critE};

\qquad {\bf for} primes $p$ in \textbf{primeSet do}

\qquad \qquad  Generate $\epsilon$ for $a = 1$ using \eqref{eqA1} 

\qquad\qquad {\bf while} $\epsilon >$ \textbf{minEpsilon do}

\qquad\qquad\qquad   Add ($\epsilon$, $p$) to \textbf{critE}

 \qquad\qquad\qquad     Increment $a$ by 1

\qquad\qquad\qquad   Generate $\epsilon$ for new $a$ using \eqref{eqA1}

\qquad\qquad {\bf end}

\qquad {\bf end}

\noindent We then implement the following algorithm to test each odd colossally abundant number in our analogue (\ref{RobinOdd})   of Robin's inequality $\frac{\sigma(N)}{N\log{\log{N}}} < \frac{e^\gamma}{2}$. Furthermore, we have that when the exponent of a prime $p$ is increased from $k$ to $k+1$ in the prime factorization of $N$, we have that $\sigma(N)$ is multiplied by a factor of $\frac{p^{k+2}-1}{p^{k+2}-p}$.

In the range of computations of these odd colossally abundant numbers,
the case of an $\epsilon$ corresponding to potentially three or more $N$ does not occur (and it is reasonable to guess that this never happens).

{\bf Result:} Generate and test all odd colossally abundant numbers.

 Initialize $N = 1$ and sum of divisors of $N$, \textbf{sumDivN}, to be 1.

 Initialize a dictionary \textbf{expForm} that keeps track of the prime factorization of $N$.

{\bf for} {\it epsilon} in \textbf{critE do}

\qquad  Increment the corresponding prime's exponent in \textbf{expForm}

 \qquad Multiply \textbf{sumDivN} by $\frac{p^{a+1}-1}{p^{a+1}-p}$ where $a$ is the new exponent of prime $p$

\qquad  {\bf if} $\frac{\text{sumDivN}}{\log{\log{N}}} > \frac{e^\gamma}{2}$ {\bf then}

\qquad\qquad{\bf return} n
 
\qquad {\bf end}

{\bf end}

\noindent Table 2 gives the list of the 23 odd colossally abundant numbers $< c_0$. Each violates the analogue  (\ref{RobinOdd})  of Robin's inequality in Theorem \ref{MainThm1}.

The next odd colossally abundant number is $$c_0 = 18565284664427130919514350125 = 3^4 \cdot 5^3 \cdot 7^2 \cdot 11\cdot 13\cdots 67,$$ which satisfies our analogue 
 (\ref{RobinOdd})  of Robin's inequality.

\begin{table}[ht]\label{OCATable}
\caption{Odd colossally abundant numbers $< c_0$.}
\begin{center}{\begin{tabular}{@{}rcl
@{}} \hline
$N$ && Factorization of $N$ \\
3 && 3 \\ 
 15 && $3 \cdot 5$ \\
 45 && $3^2 \cdot 5$ \\
 315 && $3^2 \cdot 5 \cdot 7$\\
 3465 && $3^2 \cdot 5 \cdot 7 \cdot 11$\\
 45045 && $3^2 \cdot 5 \cdot 7 \cdot 11 \cdot 13$\\
 135135 && $3^3 \cdot 5 \cdot 7 \cdot 11 \cdot 13$\\
 675675 && $3^3 \cdot 5^2 \cdot 7 \cdot 11 \cdot 13$\\
 11486475 && $3^3 \cdot 5^2 \cdot 7 \cdot 11 \cdot 13 \cdot 17$\\
 218243025 && $3^3 \cdot 5^2 \cdot 7 \dots 19$\\
 5019589575 && $3^3 \cdot 5^2 \cdot 7 \dots 23$\\
 145568097675 && $3^3 \cdot 5^2 \cdot 7 \dots 29$\\
 4512611027975 && $3^3 \cdot 5^2 \cdot 7 \dots 31$\\
 31588277195475 && $3^3 \cdot 5^2 \cdot 7^2 \cdot 11 \dots 31$\\
 94764831586425 && $3^4 \cdot 5^2 \cdot 7^2 \cdot 11 \dots 31$\\
 3506298768697725 && $3^4 \cdot 5^2 \cdot 7^2 \cdot 11 \dots 37$\\
 143758249516606725 && $3^4 \cdot 5^2 \cdot 7^2 \cdot 11 \dots 41$\\
 6181604729214089175 && $3^4 \cdot 5^2 \cdot 7^2 \cdot 11 \dots 43$\\
 290535422273062191225 && $3^4 \cdot 5^2 \cdot 7^2 \cdot 11 \dots 47$\\
 15398377380472296134925 && $3^4 \cdot 5^2 \cdot 7^2 \cdot 11 \dots 53$\\
 908504265447865471960575 && $3^4 \cdot 5^2 \cdot 7^2 \cdot 11 \dots 59$\\
 4542521327239327359802875 && $3^4 \cdot 5^3 \cdot 7^2 \cdot 11 \dots 59$\\
 \phantom{aaa}277093800961598968947975375 && $3^4 \cdot 5^3 \cdot 7^2 \cdot 11 \dots 61$\phantom{aaa}\\
 \hline
\end{tabular}}
\end{center}
\end{table}

\end{document}